\documentclass{svmult}

\usepackage[english]{babel}
\usepackage[T1]{fontenc}
\usepackage[utf8]{inputenc}

\usepackage[mathscr]{euscript}
\usepackage{mathrsfs}
\usepackage{bbold}
\usepackage{pxfonts}
\usepackage{color}
\usepackage{pb-diagram}

\renewcommand\H{\mathcal H}

\newcommand\R{\mathbb {R}}

\newcommand\e{{\sf e}}
\newcommand\g{{\mathfrak g}}

\newcommand\op{\mathsf{Op}}
\newcommand\tr{\mathrm{Tr}}
\newcommand\wG{\widehat{\sf G}}

\newcommand\G{{\sf G}}

\newcommand\Op{\mathfrak Op}
\newcommand\multiplication{\mathsf{Mult}}
\newcommand\convolution{\mathsf{Conv}}

\newcommand\fscr{\mathscr}

\begin{document}

\title*{Pseudo-differential operators associated to general type I locally compact groups}
\titlerunning{Pseudo-differential operators on groups} 
\author{M. M\u antoiu and M. Sandoval}
\institute{M. M\u antoiu \at Facultad de Ciencias, Departamento de M\'atematicas, Universidad de Chile, Las Palmeras 3645, Santiago, Chile \email{mantoiu@uchile.cl}
\and M. Sandoval \at Facultad de Ciencias, Universidad de Chile, Las Palmeras 3645,  Santiago, Chile \email{msandova@protonmail.com}}

\maketitle

\abstract{ In a recent paper by M. M\u antoiu and M. Ruzhansky, a global pseudo-differential calculus has been developed for \textit{unimodular} groups of type I. In the present article we generalize the main results to arbitrary locally compact groups of type I.  Our methods involve the use of Plancherel's theorem for non-unimodular groups.  We also make connections with a $C^*$-algebraic formalism, involving dynamical systems, and give explicit constructions for the group of affine transformations of the real line.}



\section{Introduction}\label{introduction}

Let $\G$ be a locally compact group with unitary dual $\widehat\G$\,, that is, the space of classes of unitary equivalence of (strongly continuous) unitary irreducible representations. It will be assumed that our groups are second countable and of type I. The formula (cf.~\cite{MR}, eq.~$(1.1)$ for the unimodular case)
\begin{equation}\label{eq:st}
\left[  \op(A)u \right](x) = \int_\G\!\int_{\widehat\G} \mathrm{Tr}\left(  A(x,\xi)D_\xi^{\frac 1 2} {\pi_\xi}(xy^{\scriptscriptstyle -1})^*\right)\Delta(y)^{- \frac 1 2} u(y) \, d\xi\, dy
\end{equation}
is the starting point for a global pseudo-differential calculus on $\G$\,. It involves suitable operator-valued symbols $A$ defined on $\G\times \widehat\G$\,, the modular function $\Delta$ of the group and the formal dimension operators $D_\xi$ introduced by Duflo and Moore~\cite{DM}. Formula (\ref{eq:st}) makes use of
the Haar and Plancherel measures on $\G$ and $\widehat\G$ respectively. We also fixed a measurable field of representations $(\pi_\xi)_{\xi \in \widehat\G}$ such that $\pi_\xi$ belongs to the
class $\xi$ and $\pi_\xi$ acts on a Hilbert space $\mathcal{H}_\xi$\,. 

\medskip
One of the advantages of using operator valued symbols is that one gets a global approach and a full symbol, free of localization choices, everything relying on harmonic analysis concepts attached to the group. Even for compact Lie groups there is no notion of full scalar-valued symbols for a pseudo-differential operator using local coordinates. For a more detailed discussion, for motivations and a full development of particular cases see~\cite{FR,RT,MR}.

\medskip
In the present article we are not going to rely on properties such as compactness or nilpotency nor on smoothness or unimodularity, and most hypothesis will be
on the measure theoretic side. The category of second countable type I locally compact groups has a nice integration theory and their unitary
duals have an amenable integration theory. For an introduction to this topic refer to~\cite{Fo}. This framework allows a general form of
Plancherel Theorem~\cite{DM,Fu,Tts}, which is all it is needed to develop the basic features of quantization even for a non-unimodular group.

\medskip
The interest of our extension comes mainly from the fact there are many important examples of non-unimodular groups. The simplest one is perhaps the affine group consisting of all the affine transformations of the real line, which is actually the only non-unimodular group in dimension two. In dimension three there are many infinite families of non-isomorphic non-unimodular Lie groups. Many other examples arise in the study of parabolic subgroups of semisimple Lie groups, that are used to investigate irreducible representations using extensions of Mackey's machine.

\medskip 
Formula (\ref{eq:st}) is a generalization of
the formula derived in~\cite[eq.~ (1.1)]{MR} for unimodular groups, with a difference on the order of the factors that has to do with the
choice of a convention for the Fourier transform (cf.~Remark~\ref{sec:alternative-fourier-trans}). Thus our quantization will
cover right invariant operators whereas the one in~\cite{MR} gives rise to left invariant operators. 

\medskip
Graded nilpotent Lie groups are treated systematically
in~\cite{FR} and in many other references. For a general treatment of pseudo-differential operators in a group setting, see~\cite{RT,FR}. The compact Lie groups are also treated by using global operator-valued symbols in~\cite{DR2,RT1,RT2,RTW,RW}. The recent articles containing applications and developments are too many to be cited here.
The idea of using the irreducible representations of a groups to define such calculus, seems to come from~\cite[Sect.\,1.2]{MT}, but it
was not developed in this abstract setting. All the books and articles mentioned earlier contain historical background and references to the existing
literature treating pseudo-differential operators and quantization in a group theoretic context. In many cases, specific properties of the group allow defining H\"ormander-type classes of symbols, and this has far-reaching consequences.

\medskip 
Another approach to a quantization consist of using the formalism of $C^*$- algebras. Given a locally compact group $\G$, there is an action by left (or right) translations on various
$C^*$-algebras of functions on $\G$\,. In such situations there are natural crossed products associated to them. Among the non-degenerate
representations of these $C^*$-algebras stands the Schr\"odinger representation, acting on the Hilbert space $L^2(\G)$\,. This formalism
allows to take full advantage of the theory of $C^*$-algebras in the setting of pseudo-differential operators.

\medskip
In the next two sections we introduce the notations and general harmonic analysis theory required for the quantization. In Section~\ref{sec:quantization} we make a preliminary
construction of the quantization $\op$\,. Section~\ref{sec:left-right-quant} includes a discussion on the difference and the connections between left and right quantizations, comming from the non-commutativity of the group, as well as the various $\tau$-quantizations related to ordering issues. Section~\ref{sec:some-oper-arris} exemplifies all these by showing how multiplication and convolution operators are covered by the calculus; some intricacies appear due to the presence of the modular function. The same can be said about the $C^*$-algebraic approach, which only gives a correction of the calculus ${\sf Op}$ by a factor defined by the modular function. In Section~\ref{firtanun}, for exponential groups, we put into evidence a new (but related) quantization applied to scalar symbols defined on the cotangent bundle of the group. The particular case of connected simply connected nilpotent groups has been treated in~\cite[Sect. 8]{MR}; see also~\cite{MR1} for refined results valid in the presence of flat coadjoint orbits. Other approaches in the nilpotent case can be found in~\cite{BFKG,CGGP,Me, Glo1, Glo2,Ta}.  In Section~\ref{pseudocalc}  we work out the case of the group of affine transformations of $\mathbb R$\,.

\section{General type I locally compact groups}\label{sec: Notation}

In this section we set up the general framework of the article.  

\medskip
We assume all Hilbert spaces $\mathcal{H}$ to be separable, using the convention that their scalar product $\langle \cdot,\cdot \rangle_\mathcal{H}$ is linear in the first variable and anti-linear in the second. By $\mathcal H ^\dag$ we denote the conjugate of $\H$, whose elements are the same, but the scalar product is defined as $\alpha \cdot u = \bar \alpha u$ and its inner product is conjugate to the one from $\mathcal H$, i.e. $\langle u,v \rangle_{\mathcal H^\dag} = \langle v , u\rangle_{\mathcal H}$\,. If $\pi$ is a strongly continuous unitary representation of a topological group $\G$ on a Hilbert space  $\mathcal{H}_\pi$\,, its contragradient representation $\pi^\dag$ acts on $\mathcal{H}_\pi^\dag$ by $\pi^\dag(x)f = \pi(x)f$\,; in general $\pi$ and $\pi^\dag$ are not equivalent.

\medskip
By $\mathbb{B} (\mathcal{ H})$ one denotes the $C^*$-algebra of all bounded linear operators on $\mathcal{H}$ and $\mathbb K (\mathcal{ H})$ stands for the two sided $^*$-ideal of compact operators on $\mathcal H$. We also make use of the Schatten-von Neumann classes $\mathbb {B}_p(\mathcal {H})$ for $p \geq 1$\,; these are Banach $^*$-algebras with the norm 
$\left\Vert\,T\,\right\Vert_{\mathbb{B}_p} = \mathrm{Tr}\left(( T^*T)^{p/2}\right)^{1/p}.$
The most important cases are $\mathbb{B}_1(\mathcal{H})$\,, the space of trace-class operators, and, for $p = 2$\,, the space of Hilbert-Schmidt operators, which endowed with the inner product
$\langle T,S\rangle_{\mathbb{B}_2}= \mathrm{Tr}\left(TS^*\right)$ is unitarily isomorphic with the Hilbert tensor product $\mathcal{H} \otimes \mathcal{H}^\dag$. 

\medskip
Let $\G$ be a (Hausdorff) locally compact group with unit $\e$; we also assume that it is second countable and of type I. Recall that a second countable group is separable, $\sigma$-compact and completely metrizable; in particular as a measurable space it will be standard. Also recall that a group is of type I if every factor representation is quasi-equivalent to a an irreducible one. For these groups their $C^*$-enveloping algebra are postliminal. 

\medskip
Let us fix a left Haar measure $\mu$ on $\G$\,, also denoted by $d\mu(x) =dx$\,. We get a right Haar measure $\mu^r$ defined by the formula $\mu^r(E) = \mu(E^{\scriptscriptstyle -1})$\,. Let $\Delta:\G \to (0,\infty)$ be the modular function of $\G$ satisfying $ \mu (E x) = \Delta(x) \mu(E)$ for measurable sets $E \subset\G$ and $ x\in\G$\,; this implies in particular that ${d\mu^r}= \Delta^{-1}d\mu$\,. The modular function is a continuous (smooth if $\G$ is a Lie group) homomorphism into the multiplicative group $\mathbb R_+$\,. We say that a group is \textit{unimodular} if the modular function is constant. For the convenience of the reader we recall that the modular function plays the following role on integration by substitution of variables
\begin{equation}\label{eq:change of variables}
  \int_\G f(y)\, dy = \int_\G \Delta(x) f(yx)\, dy = \int_\G\Delta(y)^{-1} f(y^{-1})\, dy\,.
\end{equation}

The spaces of $p$-integrable functions $L^p(\G) = L^p(\G,\mu)$ will always refer to the left Haar measure; these are separable Banach spaces for $p \in [1,\infty)$\,. By $C_c(\G)$ we denote the space of complex continuous functions on $\G$ with compact support, which is dense in $L^p(\G)$\,.
One has a Banach $^*$-algebra structure on $L^1(\G)$, with convolution
$$
(f * g)(x) = \int_\G f(y) g(y^{-1}x)\, dy = \int_\G \Delta(y)^{-1}f(xy^{-1}) g(y)\, dy\,,
$$
and involution given by $f^*(x) = \Delta(x)^{- 1} \overline{f (x^{-1})}$\,. In general, there is a $p$-dependent involution on $L^p(\G)$ given by
\begin{equation}\label{eq:involution}
f^*(x) = \Delta(x)^{-\frac 1 p} \overline{f (x^{ -1})}\,.
\end{equation}
But we reserve the notation $f^*$ for $p=2$\,. 

\medskip
Given a locally compact group $\G$\,, its \textit{unitary dual} $\widehat\G$ is the collection of all of its irreducible unitary representation modulo unitary equivalence. We endow $\widehat\G$ with
the Mackey Borel structure~\cite{Di}. It is known that being of type I is equivalent to $\widehat\G$ being countably separated and is also equivalent to being a standard Borel space.

\begin{example}
Some well-known examples~\cite{Fo} of type I groups are: (a) compact groups, (b) connected semisimple Lie groups, (c) Abelian groups, (d) exponentially solvable Lie groups, in particular connected simply connected nilpotent Lie groups, (e) real algebraic groups. It is known that a discrete group is of type I if and only if it possesses an Abelian normal subgroup of
finite index. 
\end{example}

For a representative $\pi_\xi \in \xi$ of an element of the unitary dual of $\G$\,, we set $\mathcal{H}_\xi = \mathcal{H}_{\pi_\xi}$. The left and right regular representations on $L^2(\G)$ are
$$
\big[\lambda_y (f)\big](x):= f(y^{ -1}x)\quad{\rm and}\quad\big[\rho_y(f)\big](x):= \Delta(y)^{\frac 1 2}f(xy)\,.
$$

We say that a standard measure $\nu$ on $\widehat\G$ is a \textit{Plancherel measure} if it yields a direct integral central decomposition of the regular
representations into irreducible representations. We do not worry to put this in formal terms, since we are only going to use some of its properties indicated in Theorem~\ref{sec:four-planch-transf}.

\medskip
Plancherel measures do exist for separable locally compact groups of type~I and in fact they are all equivalent (cf.~Theorem~\ref{sec:four-planch-transf} below). From now on, for a Plancherel measure $\nu$ we adopt the notation $d\nu(\xi) = d\xi$\,.

\medskip
There are various cases in which the Plancherel measure can be given explicitly. For Abelian groups, $\widehat\G$ is also an Abelian group (the Pontryagin dual) and the Plancherel measure coincides with one of its Haar measures. For connected, simply connected nilpotent Lie groups it corresponds to a measure on the space of coadjoint orbits arising from the Lebesgue measure on the dual $\mathfrak g^*$ of the Lie algebra. For compact groups the the Peter-Weyl theorem says that the irreducible representations form a discrete set and it describes the Plancherel measure.

\medskip
For non-unimodular groups an important role is played by \textit{the Duflo-Moore operators (also called formal dimension operators)}. They are densely defined positive self-adjoint operator with dense image $D_\pi: \mathrm{Dom}(D_\pi) \to \mathcal H_\pi$ and satisfy almost everywhere \textit{the semi-invariance condition }
\begin{equation}\label{eq:invrel}
\pi(x) D_\pi \pi(x)^* = \Delta(x)^{-1} D_\pi, \quad\forall\,x \in\G\,.
\end{equation}

In~\cite[p.~97]{Fu} an explicit construction of the operators $D_\pi$ is made for square integrable representations.  For unimodular groups, the operators $D_\pi$ are just multiplication by some positive scalar $d_\pi$, which coincides with the dimension of $\mathcal H_\pi$ when the later is finite.

\section{The Fourier and Plancherel transforms}\label{sec:fourier}

Now we recall some known results on the Fourier theory of non-unimodular groups of type I.

\medskip
Suppose we have fixed a Plancherel measure $\nu$ in $\widehat\G$\,, a measurable field of representations $(\pi_\xi)_{\xi \in \widehat\G}$ and there is a family of densely defined self-adjoint positive operators $D_\xi : \mathcal H_\xi \to \mathcal H_\xi$ satisfying (\ref{eq:invrel}) for $\nu$-almost all $\xi \in \widehat\G$\,. We define (in weak sense) the \textit{operator-valued Fourier transform} of a function $f \in L^1(\G)$ as
$$
(\mathcal{F}f)(\xi) \equiv \pi_\xi(f) = \int_\G f(y){\pi_\xi}(y) \, dy\,.
$$
The Fourier transform is a non-degenerate $^*$-representation of $L^1(\G)$\,, but in the non-unimodular case it fails to intertwine the two-sided regular representation of $\G$ with
$\int_{\widehat\G}^\oplus \xi \otimes \xi^\dag \, d\xi$ and it also fails to be an $L^2$-unitary map. So one introduces the \textit{Plancherel transform} of $f \in L^1(\G)\cap L^2(\G)$ as the operator
$$
(\mathcal Pf) (\xi) = \pi_\xi (f ) D_\xi ^{1/2}.
$$
In the following we denote by $\mathcal{P}f = \widehat f$ the Plancherel transform.

\medskip
We are going to present a partial formulation of the Plancherel Theorem for non-unimodular groups; for a proof in the case where $\G$
is unimodular we refer to~\cite{Di}. The non-unimodular Plancherel Theorem  was developed by N.~Tatsuuma in~\cite{Tts}, and latter an
extension of his theory, including some clarifications, has been obtained by Duflo and Moore~\cite{DM}.

\begin{theorem}\label{sec:four-planch-transf}
Let $\G$ be a type I second countable locally compact group. There exists a $\sigma$-finite Plancherel measure $\nu$ on $\widehat\G$\,, a measurable field of irreducible representations $(\pi_\xi)_{\xi \in \widehat\G}$ with $\pi_\xi \in \xi$\,, a measurable field $(D_\xi)_{\xi \in \widehat\G}$ of densely defined self-adjoint positive operators on $\mathcal{H}_\xi$ with dense image, satisfying (\ref{eq:invrel}) for $\nu$-almost every $\xi \in \widehat\G$\,, which have the following properties:
\begin{enumerate}
\item Let $f \in L^1(\G) \cap L^2(\G)$\,. For $\nu$-almost all $\xi \in \widehat\G$, the operator $\widehat f (\xi)$ extends to a Hilbert-Schmidt operator on $ \mathcal H_\xi$ and
$$
\left\Vert\,f\,\right\Vert ^2_2 = \int_{\widehat\G}\,\,\left\Vert\,\widehat f (\xi)\,\right\Vert_{\mathbb B_2}^2 d\xi\,.
$$
The Plancherel transformation extends in a unique way to a unitary operator
$$
\mathcal P: L^2(\G) \to \int^{\oplus}_{\widehat\G} \mathbb B _2(\mathcal H _\xi)\, d\xi\,.
$$
\item The Plancherel measure and the operator field satisfy the inversion formula
\begin{equation}\label{eq:plancherel inversion formula}
f(x) = \int_{\widehat\G} \tr \left(\widehat f (\xi)D_\xi^{\frac 1 2}{\pi_\xi}(x)^*\right)d\xi\,,
\end{equation}
for all $f$ in the Fourier algebra of $\,\G$ (see below). The inversion formula converges absolutely in the sense that $\widehat f(\xi) D_\xi^{\frac 1 2}$
extends to a trace-class operator $\nu$-a.e.~and the integral of the trace-class norms is finite.
\item Suppose there is another Plancherel measure $\nu'$ on $\widehat\G$ and measurable fields $({\pi_\xi}',{D_\xi}')_{\xi \in \widehat\G}$ that share the
properties above. Then $\nu$ and $\nu'$ are equivalent measures, and there is a measurable field of unitary operators $(U_\xi)_{\xi\in \widehat\G}\,$, intertwining $\pi_\xi$ and
${\pi_\xi}'\!$, such that for $\nu$-almost all $\xi \in \widehat\G$ the Radon-Nikodym derivative of $\nu'$ with respect to $\nu$ satisfies
$$
\frac {d\nu'} {d\nu}(\xi)\, {D_\xi}' = \, U_\xi\, D_\xi\, U_\xi^*\,. 
$$
\end{enumerate}
\end{theorem}

For simplicity we make use of the following notation
$$
\mathscr B^\oplus_2(\wG)= \int^{\oplus}_{\widehat\G} \mathbb B_2(\mathcal H_\xi) \,d\xi\,, \quad\mathscr{B}^\oplus_1(\wG)
= \int^{\oplus}_{\widehat\G} \mathbb B _1(\mathcal H_\xi) D_\xi^{- \frac 1 2}\, d\xi\,,
$$
$$
\mathscr{B}^\oplus_2(\G\times\wG)= L^2(\G) \otimes \mathscr B^\oplus_2(\wG)\,, \quad\mathscr{B}_2^\oplus(\widehat\G\times\G)= \mathscr B^\oplus_2(\wG)\otimes L^2(\G)\,.
$$
$\mathscr{B}_2^\oplus(\G\times\wG)$\,, one of the natural spaces of symbols, has
the inner product 
$$
\langle A,B \rangle_{\mathscr{B}_2^\oplus} = \int_\G \int_{\widehat\G}\mathrm{Tr} \left[A(x,\xi)B(x,\xi)^* \right]d\xi\,dx\,.
$$

Most of the results of this sections are presented in the work~\cite{Fu}. In order to shed some light on the trace-class hypothesis imposed to our symbols, we elaborate a
little on the natural domain of the Plancherel transform such that formula (\ref{eq:plancherel inversion formula}) holds. We also give the natural domain for $\mathcal P^{-1}$.

\begin{definition}
The \textit{Fourier algebra}  $A(\G)$ of a locally compact group $\G$ is 
$$
A(\G)= \{ f * g^\flat \mid f,g \in L^2(\G) \}\,,
$$
where $g^\flat(x) = \overline {g(x^{-1})}$\,. If we endow $A(G)$  with the norm
$$
\left\Vert\,u\,\right\Vert_{A(\G)} = \inf \big\{\,\left\Vert f\,\right\Vert_2\left\Vert\,{g}\,\right\Vert_2\big\vert\, u = f * g^\flat\big\}\,,
$$
it becomes a Banach $^*$-algebra with convolution and $^\flat$ as the involution (it is the space of matrix coefficient functions of the left regular representation).
\end{definition}

In~\cite[Th. 4.12]{Fu} it is shown that the Plancherel transform induces an isomorphism between the Banach spaces $A(\G)$ and $\mathscr{B}^\oplus_1(\wG)$ and an isomorphism
$$
\mathcal{P} : A(\G) \cap L^2(\G) \to \mathscr{B}^\oplus_1(\wG) \cap\mathscr{B}^\oplus_2(\wG)\,.
$$
Next proposition (cf.~\cite{Fu} Theorem 4.15) shows that $A(\G) \cap L^2(\G)$ is the natural domain of the Plancherel transform in such a way that the inversion formula holds. It also shows that in the Plancherel side, the natural domain for the inversion formula is $\mathscr{B}^\oplus_1(\wG) \cap \mathscr{B}^\oplus_2(\wG)$\,.

\begin{proposition}\label{sec:plancheinversion}
Let $F \in \mathscr{B}^\oplus_2(\wG)$ and suppose that for $\nu$-almost everywhere the operator $F(\xi) D_\xi^{\frac 1 2}$ extends to a trace-class operator. Suppose moreover that
$$
\int_{\widehat\G}\,\,\left\Vert\,F(\xi) D_\xi^{\frac 1 2}\,\right\Vert_{\mathbb{B}_1}d\xi < \infty\,.
$$
If $f$ is the inverse Plancherel transformation of $F$, then we have $\mu$-almost everywhere
\begin{equation}
f(x) = \int_{\widehat\G}\,\tr \left( F(\xi) D_\xi^{\frac 1 2} \pi_\xi(x)^* \right) d\xi\,.
\end{equation}
\end{proposition}

\section{The basic quantization}\label{sec:quantization}

In this section we introduce a quantization leading to a pseudo-differential calculus for operator valued symbols defined on the whole group involving its irreducible representation theory.  For this we fix a choice of a measurable field of representations $(\pi_\xi)_{\xi\in \widehat\G}$ and formal dimension operators $(D_\xi)_{\xi\in \widehat\G}$ such that Theorem~\ref{sec:four-planch-transf} holds. Different choices of the measurable fields of representation or Duflo-Moore operators lead to isomorphic formulations.

\medskip
For symbols $A\in L^2(\G) \otimes\big(\mathscr{B}^\oplus_1(\wG) \cap \mathscr{B}_2^\oplus(\wG)\big)$\,, set $\op(A): L^2(\G) \to L^2(\G)$ by
\begin{equation}\label{zgob}
\left[  \op(A)u \right](x) = \int_\G\int_{\widehat\G}\,\mathrm{Tr}\left( A(x,\xi)D_\xi^{\frac 1 2} {\pi_\xi}(xy^{\scriptscriptstyle -1})^*\right) \Delta(y)^{-\frac 1 2} u(y) \, d\xi\, dy\,.
\end{equation}
The operator $\op(A)$ is called the \textit{pseudo-differential operator} with symbol $A$\,. Let
$$
\textrm{ker}_A(x,y) = \Delta (y)^{- \frac 1 2}\!\int_{\widehat\G}\,\mathrm{Tr}\left(A(x,\xi)D_\xi^{\frac 1 2}\pi_\xi(xy^{-1})^* \right)d\xi\,.
$$
Since $A$ is in the domain of the inverse Plancherel transformation $\mathcal{P}_2$ in the second variable, the above integral converges absolutely and
\begin{equation}\label{dinou}
\textrm{ker}_A(x,y) = \big(\mathcal{P}_2^{-1} A\big)(x,xy^{-1}) \Delta(y)^{- \frac1 2}.
\end{equation}
By Plancherel's theorem and the change of variables given by (\ref{eq:change of variables}), we conclude that $\textrm{ker}_A$ is a square integrable function on $\G\times\G$\,; hence $\op(A)$ \textit{is a Hilbert-Schmidt operator with kernel $\textrm{ker}_A$ and Hilbert-Schmidt norm}
$$
\left\Vert\,\op(A)\,\right\Vert_{\mathbb{B}_2} = \left\Vert\,A\,\right\Vert_{\mathscr{B}^\oplus_2}.
$$
So we may extend the definition of $\op(A)$ for arbitrary symbols $A\in \mathscr{B}^\oplus_2(\G\times\wG)$\,, using the previous formula and the fact that $L^2(\G) \otimes\big(\mathscr{B}^\oplus_1(\wG) \cap\mathscr{B}_2^\oplus(\G)\big)$ is a dense subset of $\mathscr{B}^\oplus_2(\G\times\wG)$\,.

\begin{remark}
The factors $D_\xi^{\frac 1 2}$ and $\Delta(y)^{-\frac 1 2}$ disappear from (\ref{zgob}) if $\,\G$ is unimodular. By setting $\G = \mathbb R^n$, under the identification of $\,\widehat\G$ with $\mathbb R^n$ given by $\xi(x) = e^{-2 \pi i \langle \xi ,x \rangle}$, we recover the Kohn-Nirenberg calculus.
\end{remark}

\begin{remark}\label{pro: Op is unitary}
Let us define by $\Lambda_{u,v}(w) = \langle w,u \rangle v\,,\,\forall\,w\in L^2(\G)$\,,
the rank-one operator associated to the pair $(u,v)$. One shows easily that $\Lambda_{u,v}\!=\!\op(\mathcal V_{u,v})$ for {\rm the Wigner transform of $(u,v)$}
\begin{eqnarray}
\mathcal V_{u,v}(x,\xi) = \int_\G \Delta(y^{\scriptscriptstyle -1}x)^{\frac 1 2}\overline{u(y^{\scriptscriptstyle -1}x)}v(x)  {\pi_\xi} (y)D_\xi^{\frac 1 2}\,dy\,.\nonumber
\end{eqnarray}
\end{remark}

\begin{remark}\label{stiuka}
The fact that $\op$ is an isomorphism allows us to define a product which we will call the \textit{Moyal product}, and an involution on $\mathscr{B}(\G\times\wG)$ by 
$$
\op(A\#B) = \op(A) \op(B)\quad{\rm and}\quad\op(A^\#) = \op(A)^*.
$$
\end{remark}

\begin{remark} 
We also note that the left regular representation of $\G$ induces a representation acting on $\mathscr{B}^\oplus_2(\G\times\wG)$\,. Let $A\in\mathscr{B}^\oplus_2(\G\times\wG)$ be a symbol and $y \in\G$\,, then $\lambda_y \op(A)\lambda_y^*$ is a Hilbert-Schmidt operator, hence there is some other symbol $y.a\in\mathscr{B}^\oplus_2(\G\times\wG)$ such that
$$
\lambda_y \op(A)\lambda_y^* = \op(y.A)\,.
$$
It is easy to see that this is an action of the group, composed of a left translation in the  first variable and a unitary equivalence in the representation space: 
$$
(y.A)(x,\xi) = \pi_\xi (y)A(y^{-1}x,\xi)\pi_\xi (y)^*.
$$
\end{remark}

For compact Lie groups, in~\cite{RT}, and for graded nilpotent groups, in~\cite{FR} (see in both cases references therein), H\"ormander-type classes of functions have been developed for the global operator-valued pseudo-differential calculus. This leads to many interesting applications. Such a task is more difficult for larger classes of groups, and impossible in the generality of the present paper. We indicate, however, a way to extend the quantization procedure; it is still useful, but much less effective than the $S^m_{\rho,\delta}$-formalism, which is not available without strong extra structure.

\medskip
Assume for simplicity that $\G$ is a type I Lie group. One sets $\mathcal D(\G):=\mathcal C^\infty_{\rm c}(\G)$ with the usual inductive limit topology. The strong dual of $\mathcal D(\G)$\,, denoted by $\mathcal D'(\G)$\,, is composed of distributions. Let us define 
$$
\mathscr D(\wG):={\fscr F}[\mathcal D(\G)]\subset\mathscr B^\oplus_2(\wG)\cap\mathscr B^\oplus_1(\wG)
$$ 
with the locally convex topological structure transported from $\mathcal D(\G)$ and then the projective tensor product
$$
\mathscr D\big(\G\times\wG\big):=\mathcal D(\G)\,\overline\otimes\,\mathscr D(\wG)\subset\mathscr B^\oplus_2(\G\times\wG)\,.
$$
Also using its strong dual, one gets a Gelfand triple 
$$
\mathscr D\big(\G\times\wG\big)\hookrightarrow\mathscr B^\oplus_2(\G\times\wG)\hookrightarrow\mathscr D'\big(\G\times\wG\big)\,.
$$
Then the pseudo-differential calculus $\,{\sf Op}:\mathscr B^\oplus_2(\G\times\wG)\rightarrow\mathbb B^2\big[L^2(\G)\big]$
\begin{itemize}
\item
restricts to an isomorphism $\,{\sf Op}:\mathscr D(\G\times\wG)\rightarrow\mathbb B\big(\mathcal D'(\G);\mathcal D(\G)\big)$\,,
\item
extends to an isomorphism $\,{\sf Op}:\mathscr D'(\G\times\wG)\rightarrow\mathbb B\big(\mathcal D(\G);\mathcal D'(\G)\big)$\,.
\end{itemize}

We have denoted above by $\mathbb B(\mathcal A;\mathcal B)$ the space of all linear continuous mappings between the locally convex vector spaces $\mathcal A$ and $\mathcal B$\,. The proof is an easy adaptation of the proof of the corresponding result in~\cite[Sect. 5]{MR} and relies on the form (\ref{dinou}) of the kernel and on Schwartz's Kernel Theorem. The partial inverse Plancherel tranform $\mathcal P_2^{-1}$ is taken into account by the definition of $\mathscr D(\wG)$\,, while the change of variables $y\mapsto xy^{-1}$ and multiplication by $\Delta^{-1/2}$ are $\mathcal D$-isomorphisms.

\begin{remark}
If $\G$ is not a Lie group, one still can perform the same extension procedure using the Bruhat spaces $\mathcal D(\G)$ and $\mathcal D'(\G)$ introduced in~\cite{Br}. This has been done in~\cite[Sect. 5]{MR} for unimodular groups, but one can adapt everything to the non-unimodular setting. We only sketched here the Lie case for space reasons and because, any way, this is the most important.
\end{remark}

\section{Other quantizations}\label{sec:left-right-quant}

Having in mind the familiar Kohn-Nirenberg quantization for $\G = \mathbb R^n$, one notes that for non abelian groups there are two possible generalizations, connected to the non-commutativity in $\mathbb B(\H)$\,: a left quantization $\op_L\equiv\op$\,, the one used so far, and a right quantization 
$$
\left[  \op_R(a)u \right](x) = \int_\G\int_{\widehat\G} \mathrm{Tr}\left(a(x,\xi)D_\xi^{\frac 1 2} {\pi_\xi}(y^{\scriptscriptstyle -1}x)^*\right)\Delta(xy^{\scriptscriptstyle -1})^{\frac 1 2} u(y) \, d\xi\, dy\,. 
$$
Actually, this two quantizations are equivalent in the following sense: 

\begin{proposition}\label{grozav}
Let $A$ be a symbol, and consider the symbol defined by 
$$
\tilde A(x,\xi) = \pi_\xi(x)^* A(x,\xi)\pi_\xi(x)\,.
$$
Then $\op_L(A) = \op_R(\tilde A)$\,. 
\end{proposition}

\begin{proof}
The result follows from the following computation
\begin{eqnarray}
\mathrm{Tr}\left(\tilde A(x,\xi)D_\xi^{\frac 1 2} {\pi_\xi}(y^{\scriptscriptstyle -1}x)^*\right)&=&\mathrm{Tr}\left(\pi_\xi(x)^*A(x,\xi)\pi_\xi(x)D_\xi^{\frac 1 2} {\pi_\xi}(y^{\scriptscriptstyle -1}x)^*\right)\nonumber\\
&=&\Delta(x)^{-\frac 1 2}\mathrm{Tr}\left(\pi_\xi(x)^*A(x,\xi)D_\xi^{\frac 1 2}\pi_\xi(x) {\pi_\xi}(y^{\scriptscriptstyle -1}x)^*\right)\nonumber\\
&=&\Delta(x)^{-\frac 1 2}\mathrm{Tr}\left(A(x,\xi)D_\xi^{\frac 1 2}{\pi_\xi}(xy^{\scriptscriptstyle -1})^*\right)\,,\nonumber
\end{eqnarray}
based on the properties of the trace, of the representation $\pi_\xi$ and on the semi-invariance formula (\ref{eq:invrel}).
\end{proof}

\begin{remark}\label{tracas}
The usual pseudo-differential calculus in $\R^n$ disposes of an extra parameter $\tau\in[0,1]$\, connected to ordering issues in the quantization, arising from the non-commutativity of the operators (positions, momenta) that are behind its definition. This can also be implemented in our general situation (both for the left and for the right quantization), using a (any!) measurable function $\tau:\G\to\G$\,. In favorable cases there is even a symmetric quantization, having special properties, analog to the Weyl calculus ($\tau=1/2$) for the particular group $\G=\R^n$. This has been explained in~\cite{MR}, it can be extended to our non-unimodular groups, but we are not going to indicate here the easy adaptations.
\end{remark}

One encounters in the next section formulae showing that left (respectively right) convolution operators emerge naturally from the left (respectively right) quantization. The $\tau$ parameter deals (in the extreme cases $\tau(x)=\e$ and $\tau(x)=x$) with setting multiplication operators to the left or to the right of convolution operators.

\begin{remark}\label{sec:alternative-fourier-trans}
Alternatively, we could also define the Plancherel transform as
$$
[\dot\mathcal{P}(f)](\xi) = D_\xi ^{\frac 1 2} \int_\G f(x) {\pi_\xi} (x)^*\, dx\,.
$$
For unimodular groups, this was the choice in~\cite{MR} and it lead to somehow different formulae than the present ones, so we indicate briefly the relation. By the semi-invariance relation and the involution (\ref{eq:involution}) for $p=2$\,, one has
\begin{eqnarray}
&[\dot\mathcal{P}(f)](\xi)=\int_\G f(x){\pi_\xi}(x)^* D_\xi^{\frac 1 2} \Delta(x)^{-\frac 1 2}\,dx \nonumber \\
&= \int_\G f(x^{\scriptscriptstyle -1})\Delta(x)^{-\frac 1 2} {\pi_\xi}(x)D_\xi^{\frac 1 2}\,dx = [\mathcal{P} ({\bar{f}^*})](\xi) \nonumber\,.
\end{eqnarray}
\textit{So the two definitions differ by an automorphism $f\mapsto{\bar{f}^*}$ of $L^2(\G)$\,.} Another thing to have in mind is that the inversion formula (\ref{eq:plancherel inversion formula}) this time reads
\begin{eqnarray}
f(x) = \bar{f}^*(x^{-1}) \Delta(x)^{-\frac 1 2} 
= \int_{\widehat\G} \tr\left(\mathcal{P} (\bar{f}^*) D_\xi^{\frac 1 2}\pi_\xi(x^{-1})^*\right)  \Delta(x)^{-\frac 1 2}\, d\xi \nonumber\\
= \int_{\widehat\G} \tr \left( D_\xi^{\frac 1 2} \dot{\mathcal{P}} (f)\pi_\xi(x)\right) \, d\xi\,.\nonumber
\end{eqnarray}
\end{remark}

\section{Some operators arising from the calculus}\label{sec:some-oper-arris}

Important families of operators in $L^2(\G)$ are formed of multiplication and convolution operators. We show now how to recover
them using our pseudo-differential calculus; the non-unimodular case has some particular features, due to the presence of the modular function and of the formal dimension operators.

\medskip
For two square integrable functions $f,g \in L^2(\G)$ we define the operators
\begin{eqnarray}
\big[\multiplication_f(u)\big](x) = f(x)u(x)\,,  \nonumber\\
\big[\convolution_g^L (u)\big](x) = \int_\G g(y) u(y^{-1}x)\,dy\,. \nonumber
\end{eqnarray}

\begin{remark}
In general $\multiplication_f$ and $\convolution_g^L$ are not bounded in $L^2(\G)$\,. In fact $\multiplication_f$ is bounded if and only if $f$ is essentially bounded and $\convolution_g^L$ is bounded if and only if $\,{\rm ess\,sup}_{\xi \in \widehat\G}\,\left\Vert\,[\mathcal F(g)](\xi)\,\right\Vert\!<\!\infty$\,. The second assertion follows from the easy formula (note that here $\mathcal F(g)$ is relevant, not $\widehat g=\mathcal P(g)$!)
$$
\mathcal P\circ\convolution_g^L=\int_{\wG}^\oplus[\mathcal F(g)](\xi)d\xi\circ\mathcal P\,.
$$
\end{remark}

But for unimodular groups the composition $\multiplication_f\,\convolution_g^L$ does extend to a Hilbert-Schmidt operator. Assume $\G$ unimodular (so $\mathcal F$ and $\mathcal P$ coincide), let $f,g \in L^2(\G)$ and define the symbol $A$ by
$$
A(x,\xi) =  f(x)\, \widehat g (\xi)\,.
$$
Using Plancherel inversion formula one gets $\op_L(A) u = f \cdot (g * u)$\,. For non-unimodular groups the picture changes dramatically. The operators $\multiplication_f\circ \convolution_g$ are no longer Hilbert-Schmidt; in fact one has that
$$
\left\Vert\,{\multiplication_f\,\convolution_g^L }\,\right\Vert_{\mathbb{B}_2} = \left\Vert{\,\Delta^{-\frac 1 2} f }\,\right\Vert_2\left\Vert\,\Delta^{\frac 1 2} g\,\right\Vert_2.
$$
In general $\multiplication_f\,\convolution_g$ is not even a bounded operator if $f$ and $g$ are not chosen in a suitable manner.

\medskip
One way to fix this is taking functions in appropriate dense subspaces. Choose $f,g \in L^2(\G)$ such that the functions $\Delta^{-\frac 1 2}f,\,\Delta^{\frac 1 2}g$ are square integrable and set
$$
A(x, \xi) = \big(\Delta^{-\frac 1 2}f\big)(x) \widehat{\big(\Delta^{\frac 1 2} g\big)}(\xi)\,.
$$
Then for $u \in L^2(G)$ we have $\Op(A)u = f \cdot ( g * u )$\,, which may be written
$$
\multiplication_f\,\convolution_g^L=\op\big[\big(\Delta^{-\frac 1 2}f\big)\otimes\big(\widehat{\Delta^{\frac 1 2}g}\big)\big]\,.
$$ 
Indeed, by Plancherel inversion
\begin{eqnarray}
[\op(A)u ](x)&= \int_\G \Delta(x)^{- \frac 1 2} f(x)\Delta(xy^{-1})^{\frac 1 2}g(xy^{-1})\Delta(y)^{-\frac 1 2} u(y)\,dy   \nonumber \\
&=  f(x) \int_\G \Delta(y)^{-1}g(xy^{-1})u(y)\, dy = f(x)( g * u )(x)\,. \nonumber
\end{eqnarray}
Another way to express the relation between symbols of the form $A = f \otimes \widehat g$ and operators of multiplication and
convolution is given in the following formulas
\begin{eqnarray}
\op_L\big(f\otimes \widehat g\,\big) = \multiplication_f\,\convolution_g^L\,\multiplication_{\Delta^{1/2}} = \multiplication_{\Delta^{1/2}f}\,\convolution_{\Delta^{-1/2}g}^L\,,\nonumber\\
\label{sec:right-quant-conv-form}
\op_R\,\big(f\otimes \widehat g\,\big) = \multiplication_f\,\convolution_{\Delta^{1/2}g}^R = \multiplication_{\Delta^{1/2}f}\,\convolution_g^R\,\multiplication_{\Delta^{-1/2}}\,;\nonumber
\label{sec:left-quant-conv-form}
\end{eqnarray}
here $\convolution_g^R$ is the operator given by $\convolution_g ^R(u) = u * g$\,.

\medskip
There are other convolution operators that appear in the literature. In~\cite[Sect.\,1.2]{De} the author introduces (right) convolution operators by
$$
\big[\check \convolution_g^R(u)\big](x) = \int_\G u(xy)\Delta(y)^{\frac 1 2} g(y)\, dy =(u*\bar{g}^* )(x)\,.
$$
These operators are then used to study the space of left invariant operators. We leave to the reader the task of finding the relevant connections with our left and right quantizations. Anyhow, when studying these operators in $L^p$-spaces, corrections of the powers of the factors $\Delta$ seem natural and useful.

\section{The $C^*$-algebraic formalism}\label{sec:crossed-product-c}

We introduce first some tools from the theory of crossed products of $C^*$-algebras.

\begin{definition}
  A \emph{$C^*$-dynamical system}\index{$C^*$-dynamical system} is a triplet $(\mathcal{A},\G,\alpha)$\,, where $\G$ is a locally compact group, $\mathcal{A}$ is a $C^*$-algebra and $\alpha:\G \to \mathrm{Aut}({\mathcal A})$ is a strongly continuous representation of $\G$\,.
\end{definition}

To such a $C^*$-dynamical system we associate the space $L^1(\G;\mathcal{A})$ of all Bochner-integrable functions $F:\G\to \mathcal{A}$\,; it is a Banach $^*$-algebra with laws 
\begin{eqnarray}
  (F \star G)(x) = \int_\G F(y)\,\alpha_y \!\left( G(y^{-1}x) \right)dy\,, \nonumber\\
  F^\star(x) = \Delta(x)^{-1}\alpha_x \!\left( F(x^{-1})^*\right). \nonumber
\end{eqnarray}
The Banach $^*$-algebra $L^1(G;\mathcal{A})$ is naturally isomorphic to the projective tensor product $\mathcal{A} \otimes L^1(\G)$\,. Consider the \emph{universal norm}\index{universal norm} on $L^1(\G;{\mathcal A})$ given by
$$
\left\Vert\,F\,\right\Vert_{\mathcal{A} \rJoin\G} = \sup_\rho  \left\Vert{\,\rho(F)\,}\right\Vert ,
$$
where the supremum is taken over all non-degenerate $^*$-representations.  The \emph{crossed product}\index{crossed product of $C^*$-algebras}
$\mathcal{A} \rJoin\G$ is the enveloping $C^*$-algebra of $L^1(\G;\mathcal{A})$\,, that is, its completion under the norm $\left\Vert{\cdot }\right\Vert_{\mathcal{A} \rJoin\G}$\,.

\begin{example}
Let $\mathcal A$ be a $C^*$-algebra, $\G=\{\e\}$ the trivial group and $\alpha$ the trivial representation; then $\mathcal A \rJoin\G$ is naturally isomorphic to $\mathcal A$\,. The group $C^*$-algebra $C^*(\G)$ is obtained taking $\mathcal A=\mathbb C$\,. 
\end{example}

\begin{definition}
A \emph{covariant representation}\index{covariant representation} of a $C^*$-dynamical system $(\mathcal{A},\G,\alpha)$ is composed of a unitary representation $\pi$ of $\G$ and a non-degenerate $^*$-representation $\rho$ of $\mathcal{A}$\,, both acting on a Hilbert space $\mathcal{H}$, in such a way that they satisfy the relation
$$
\pi(x) \rho (f) \pi (x) ^* = \rho(\alpha_x f)\,, \qquad f \in {\mathcal A}\,, x \in\G\,.
$$
We denote this data as the triplet  $(\rho,\pi,\mathcal{H})$\,.
\end{definition}

\begin{example}\label{radeberger}
The most interesting example is that attached to a continuous action of $\,\G$ by homeomorphisms of a locally compact space $\Omega$\,; this induces an action on $\mathrm{ C_0(\Omega)}$ given by $[\alpha_x(f)](\omega) = f(x^{-1}\!\cdot\omega)$\,. Then $(C_0(\Omega),\G,\alpha)$ is a $C^*$-dynamical system and it encapsulates all the information of the group action. A covariant representations is the same as a system of imprimitivity~\cite{Mac} Sect.\,3.7. In fact there is a one-to-one correspondence between topological actions of a group $\G$ and $C^*$-dynamical systems $(\mathcal{A}\,,\G,\alpha)$ where the $C^*$-algebra $\mathcal{A}$ is Abelian; this can be easily seen from the fact that every abelian $C^*$-algebra is of the form $C_0(\Omega)$ for some locally compact space and there is a one to one correspondence between strongly continuous representations $\alpha :\G \to \mathrm{Aut}(C_0(\Omega))$ and continuous actions of $\G$ on $\Omega$~\cite[Proposition 2.7]{Wi}.
\end{example}

Any covariant representation $(\rho,\pi,\mathcal{H})$ of a $C^*$-dynamical system induces a non-degenerate $^*$-representation $\rho \rJoin \pi$ of the crossed product $\mathcal{A} \rJoin\G$ on $\mathcal{H}$, the unique extension of the representation of $L^1(G;\mathcal{A})$ given by 
$$
(\rho \rJoin \pi)(F)  = \int_\G \rho\left(F(y)\right)\pi(y)\,dy\,.
$$
This process sets up a bijection between the covariant representation of a $C^*$-dynamical system and the non-degenerate $^*$-representations of the
crossed product associated to it~\cite[Proposition 2.40]{Wi}.

\medskip
There is a natural covariant representation associated to any left-invariant $C^*$-algebra of functions defined on $\G$\,. Let $\mathcal{A}$ be a left-invariant $C^*$-subalgebra of the space of bounded left uniformly continuous functions on $\G$\,. For an $\mathcal{A}$-valued function $F$ on $\G$ and elements $x,z \in\G$ we make
the convenient identification $[F(x)](z) = F(z,x)$\,. The triplet $(\mathcal{A},\G,\alpha)$ is a $C^*$-dynamical system when endowed with the
action $\alpha:\G \to \mathrm{Aut}(\mathcal{A})$ given by $[\alpha_x(F)](z,y) =F(z,x^{-1}y)$\,.  Then our convolution and involution laws are given by 
\begin{eqnarray}
(F \star G)(z,x) = \int_G F(z,y)\, G(y^{-1}z,y^{-1}x)\,dy\,,\nonumber \\
F^\star(z,x) = \Delta(x)^{-1}\overline{F(x^{-1}z,x^{-1})}\,. \nonumber
\end{eqnarray}
In $\mathcal{H}=L^2(\G)$ we have a covariant representation of $(\mathcal{A},\G,\alpha)$ given by
\begin{eqnarray}
[\lambda_x(u)](y) = u (x^{-1}y)\,,\quad [\multiplication_f(u)](y)= f(y) u(y)\,. \nonumber
\end{eqnarray}
The \emph{Schr\"odinger representation}\index{Schr\"odinger representation} is the integrated form ${\sf Sch} = \multiplication \rJoin \lambda$ of
$\mathcal{A} \rJoin\G$\,; more explicitly, for a function $F \in L^1(\G;\mathcal{A})$
\begin{equation}\label{dinger}
[{\sf Sch}(F)u](x)= \int_\G F(x,y)u(y^{-1}x)\, dy= \int_\G F(x,xy^{-1})\Delta(y)^{-1}u(y)\,dy\,.
\end{equation}
One gets an integral operator ${\sf Sch}(F)={\sf Int}(L_F)$ with kernel
$$
L_F(x,y)=F(x,xy^{-1})\Delta(y)^{-1}.
$$
Using (\ref{eq:change of variables}) and the left invariance of the Haar measure, one gets
$$
\int_\G\!\int_\G |\,L_F(x,y)\,|^{\,p}\,dxdy=\int_\G\!\int_\G |\,F(x,y)\,|^{\,p}\Delta(x^{-1}y)^{p-1}dxdy\,.
$$
Only for $p=1$ the correspondence $F\mapsto L_F$ is an $L^p(\G\times\G)$-isometry. This has consequences upon the relevance of the Schr\"odinger representation of the crossed product for our pseudo-differential calculus in the non-unimodular case. Extending the procedure in~\cite[Sect. 7]{MR}, one would like to define
$$
\mathfrak{Op}(A)={\sf Sch}\big[\mathcal P_2^{-1}(A)\big]\,,
$$
composing the Schr\"odinger representation with the unitary inverse Plancherel
transformation $\mathcal P^{-1}_2:\mathscr B^\oplus_2(\G\times\wG)\to L^2(\G\times\G)$\,. A
direct computation leads to
$$
\left[\mathfrak{Op}(A)u \right](x) = \int_\G\int_{\widehat\G}\,\mathrm{Tr}\left(A(x,\xi)D_\xi^{\frac 1 2} {\pi_\xi}(xy^{\scriptscriptstyle -1})^*\right) \Delta(y)^{-1} u(y) \, d\xi dy\,,
$$
wich differs from (\ref{zgob}) by a factor $\Delta(y)^{-1/2}$. Thus one has 
\begin{equation}\label{folosit}
\mathfrak{Op}(A)=\op(A)\circ{\sf Mult}_{\Delta^{-1/2}}\,,
\end{equation}
with a non-trivial correction if $\G$ is not unimodular, loosing in general the good square integrability properties of $\op$\,.

\begin{remark}
It is still legitimate to study the quantization $A\mapsto\mathfrak{Op}(A)$ even for groups that are not unimodular. Involving arbitrary  left-invariant $C^*$-algebras $\mathcal A$ of functions defined on $\G$\,, this allows a direct study of "pseudo-differential operators with coefficients of type $\mathcal A$\,". In addition, having the crossed product construction in the background, it allows extending the results concerning essential spectra and Fredholm properties of global pseudo-differential operators, that have been obtained in~\cite{Ma} only when $\Delta=1$\,. The key role is played by the Gelfand spectrum $\Omega(\mathcal A)$ of $\mathcal A$ which, under suitable assumptions, is a compactification of $\G$ on which $\G$ acts continuously. The relevant spectral information is contained in the orbit structure of the boundary $\Omega(\mathcal A)\setminus\G$\,.
\end{remark}

\begin{remark}
Adding a $2$-cocycle to the formalism and studying twisted operators is also possible, as done in~\cite{BM} for unimodular groups. For $\G=\R^n$ this reproduces the gauge covariant magnetic Weyl calculus~\cite{IMP,MP,MPR}.
\end{remark}

\begin{remark}
One can also extend~\cite[Subsect. 7.4]{MR}, in which covariant families of pseudo-differential operators are introduced  starting from suitable symbols associated to general Abelian $C^*$-dynamical systems, as in Example~\ref{radeberger}. This is a natural way to generate interesting classes of random Hamiltonians.
\end{remark}

\section{The case of exponential Lie groups}\label{firtanun}

Let $\G$ be an exponential group with Lie algebra $\mathfrak g$ and exponential map ${\sf exp}:\mathfrak g\to\G$ (a diffeomorphism, by definition) with inverse ${\log}:\G\to\mathfrak g$\,. Such a group is second countable and type I, so it fits in our setting.

\medskip
Let us set $\sigma:={\log}(\mu)$ for the image through ${\log}$ of the (fixed) Haar measure on $\G$\,.
It is known how it is related to the Lebesgue measure on $\mathfrak g$\,: in terms of the adjoint action ${\rm ad}:\g\to{\rm aut}(\g)$ one has $d\sigma(X)=\theta(X)dX$\,, where
$$
\theta(X)=\Big\vert\det\frac{1-e^{-{\rm ad}_X}}{{\rm ad}_X}\,\Big\vert\,.
$$
Thus we have a unitary operator
$$
U_\theta:L^2(\mathfrak g;\theta(X)dX)\to L^2(\mathfrak g;dX)\,,\quad U_\theta(v):=\theta^{1/2}\,v\,.
$$
Consequently, one also gets the unitary operators
$$
{\sf Exp}:L^2(\G)\to L^2(\mathfrak g;\theta(X)dX)\,,\quad {\sf Exp}(u):=u\circ{\sf exp}\,,
$$
$$
{\sf Exp_\theta}:L^2(\G)\to L^2(\mathfrak g;dX)\,,\quad {\sf Exp}_\theta(u):=(U_\theta\circ{\sf Exp})(u)=\theta^{1/2}(u\circ{\sf exp})\,.
$$
Let $\mathfrak g^*$ be the dual of the Lie algebra. There is a unitary Fourier transformation $\,\mathcal F\!_{\mathfrak g\mathfrak g^*}:L^2(\mathfrak g;dX)\rightarrow L^2(\g^*;d\mathcal X)$ associated to the duality 
$$
\langle\cdot\!\mid\!\cdot\rangle:\g\times\g^*\rightarrow\mathbb R\,,\;\langle Y\!\!\mid\!\!\mathcal X\rangle:=\mathcal X(Y)\,.
$$ 
It is defined by
$$
\big[\mathcal F\!_{\mathfrak g\mathfrak g^*}(v)\big](\mathcal X)=\int_{\g}e^{-i\langle X\mid \mathcal X\rangle}v(X)dX\,,
$$
with inverse given (for a suitable normalization of $d\mathcal X$\,) by
$$
\big[\mathcal F^{-1}_{\mathfrak g\mathfrak g^*}(w)\big](X)=\int_{\g^*}\!e^{i\langle X\mid \mathcal X\rangle} w(\mathcal X)d\mathcal X.
$$
Then the most important transformation
$$
\mathcal F_{\G\mathfrak g^*}=\mathcal F\!_{\mathfrak g\mathfrak g^*}\!\circ{\sf Exp_\theta}:L^2(\G)\to L^2(\mathfrak g^*;d\mathcal X)
$$
is given explicitly by
$$
\big[\mathcal F_{\G\mathfrak g^*}(u)\big](\mathcal X)\!=\!\!\int_\mathfrak g e^{-i\langle X\mid \mathcal X\rangle}u[{\exp(X)}]\,\theta(X)^{\frac 1 2}dX\!=\!\!\int_\G\!e^{-i\langle{\log}(x)\mid \mathcal X\rangle}u(x)\theta[{\log}(x)]^{-\frac 1 2}dx\,,
$$
with inverse
$$
\Big[\mathcal F_{\G\mathfrak g^*}^{-1}(w)\Big](x)\,=\,\theta[{\log}(x)]^{-1/2}\!\int_{\mathfrak g^*} e^{i\langle{\log}(x)\mid \mathcal X\rangle}w(\mathcal X)\,d\mathcal X.
$$
Thus, recalling the unitarity of the Plancherel transformation, the operator 
$$
\mathcal L:=\mathcal P\circ\mathcal F_{\G,\mathfrak g^*}^{-1}:L^2(\mathfrak g^*;d\mathcal X)\to \mathscr B^\oplus_2(\wG)
$$
is also unitary. Just by using the definitions, one has 
$$
[\mathcal L(w)](\xi)=\int_\G\!\int_{\mathfrak g^*} e^{i\langle{\log}(x)\mid \mathcal X\rangle}\theta[{\log}(x)]^{-\frac 1 2}w(\mathcal X)\pi_\xi(x)D_\xi^{\frac 1 2}dx d\mathcal X,
$$
with inverse given explicitly on $\mathscr B^\oplus_1(\wG)\cap\mathscr B^\oplus_2(\wG)$ by
$$
\big[\mathcal L^{-1}(v)\big](\mathcal X)=\int_\G\!\int_{\wG}  e^{-i\langle{\log}\,x\mid \mathcal X\rangle}\theta[{\log}(x)]^{-\frac 1 2}{\rm Tr}_\xi\big[ v(\xi)D_\xi^{\frac 1 2}\pi_\xi(x)^*\big]\,dxd\xi\,.
$$

We make a notational convention: If $\,\mathcal T:\mathcal M\to\mathcal N$ is a linear transformation between Hilbert spaces, we denote by 
$$
\mathbf T\equiv{\sf id}\otimes\mathcal T:L^2(\G)\otimes\mathcal M\to L^2(\G)\otimes\mathcal N
$$ 
the obvious operator. Taking into account the unitary isomorphism $\mathcal L$\,, as well as the unitary pseudodifferential calculus $\op:L^2(\G)\otimes\mathscr B^\oplus_2(\wG)\to\mathbb B_2[L^2(\G)]$\,, one defines
$$
{\sf op}=\op\circ\mathbf L=\op\circ\mathbf P\circ\mathbf F_{\G\mathfrak g^*}^{-1}:L^2(\G)\otimes L^2(\mathfrak g^*;d\mathcal X)\to\mathbb B_2[L^2(\G)]\,,
$$
seen as an attempt to quantize the cotangent bundle $T^*\G\cong\G\times\g^*$. By construction, ${\sf op}$ is a unitary transformation. The simplest way to get an explicit form, is to recall formula (\ref{folosit}):
$$
{\sf op}({\sf B})=\op[\mathbf L({\sf B})]={\sf Sch}\big(\mathbf P^{-1}[\mathbf L({\sf B})]\big)\circ\multiplication_{\Delta^{1/2}}={\sf Sch}\big(\mathbf F^{-1}_{\G\g^*}({\sf B})\big)\circ\multiplication_{\Delta^{1/2}}\,,
$$
which together with (\ref{dinger}) leads to
$$
[{\sf op}({\sf B})u](x)=\int_\G\!\int_{\g^*}e^{i\langle{\log}(xy^{-1})\mid \mathcal X\rangle}\theta[{\log}(xy^{-1})]^{-\frac 1 2}{\sf B}(x,\mathcal X)\Delta(y)^{-\frac 1 2}dyd\mathcal X.
$$

\section{The affine group}\label{pseudocalc}

In this section we indicate basic formulae for pseudo-differential calculi on the affine group of the real line. The theory of unitary representations of the Affine group has been worked out by Gelfand and Na\"imark; see~\cite[Sect 6.7]{Fo}. In this section
$$
\G = \{ (b,a) \in \mathbb R^2 \mid a>0\}=\R\times\R_+\,,
$$
denotes the affine group, with product law
$$
(b,a)\cdot (b',a') = (ab'+b,aa')\,.
$$
The group $\G$ is a type I solvable Lie group, it is an obvious semi-direct product, and it is the connected component of the identity of a similar connected Lie group for which the restriction is only $a\ne 0$\,. The left Haar measure is $a^{\scriptscriptstyle -2}dadb$\,, and its right Haar measure is $a^{\scriptstyle-1}dadb$\,, hence the modular function is given by
$\Delta(b,a) = a^{\scriptscriptstyle -1}.$

\medskip
Let $\mathfrak g \equiv \mathbb R^2$ be the Lie algebra of $\G$\,, with bracket defined by
$$
[(\beta,\alpha),(\beta',\alpha')] = (\alpha \beta'- \alpha' \beta,0)\,.
$$
An easy computation gives
$$
\theta(\beta,\alpha)=\frac{1-e^{-\alpha}}{\alpha}\,.
$$
The exponential map 
$$
{\sf exp}: \mathfrak g \to\G\,,\quad {\sf exp}(\beta,\alpha) = \Big( \frac \beta \alpha (e^\alpha - 1) , e^\alpha\Big)
$$ 
is a diffeomorphism. Its inverse is (a limit is necessary for $a = 1$)
$$
{\log}(b,a) = \Big(\frac b {a-1} \log(a), \log(a)\Big)\,.
$$

Thus we have all the elements needed to express the quantization. With the preliminary computation
$$
{\log}\big[(b,a)(b_1,a_1)^{-1}\big]={\log}\Big(b-\frac{ab_1}{a_1},\frac{a}{a_1}\Big)=\Big(\frac{a_1b-ab_1}{a-a_1}\log\big(\frac{a}{a_1}\big),\log\big(\frac{a}{a_1}\big)\Big)\,,
$$
one arrives at
\begin{equation}
\begin{array}{r c l}
[{\sf op}({\sf B})u](b,a)&=&\int_\R\!\int_{\R_+}\!\int_\R\!\int_\R e^{i\log(\frac{a}{a_1})[\mathfrak x+\frac{a_1b-ab_1}{a-a_1}\mathfrak y]}{\sf B}\big((b,a),(\mathfrak y,\mathfrak x)\big)\\
 & \ & \quad  u(b_1,a_1)\,\Big[\frac{a}{a-a_1}\log\big(\frac{a}{a_1}\big)\Big]^{1/2}\!a_1^{-3/2}da_1db_1d\mathfrak y d\mathfrak x\,.
\end{array}
\end{equation}
One of the interesting properties of $\G$ is that its unitary dual
consists only of two points with positive Plancherel measure equal to
$1$ each and a $\nu$-null set of one-dimensional representations~\cite[Sect. 6.7]{Fo}, that we afford neglecting. Setting
$\R_-=-\R_+=(-\infty,0)$ and $\H_\pm=L^2(\R_\pm;ds)$\,, the two irreducible
representations admit the realization 
$$
\pi_\pm:\G\to\mathbb B(\H_\pm)\,,\quad[\pi_\pm(b,a)\varphi](s) = a^{1/2}e^{2\pi i bs} \varphi(a s)\,.
$$
Of course, they are the restrictions to $\H_\pm$ (respectively) of an obvious representation on $L^2(\R)$\,.
We denote the equivalence class of $\pi_\pm$ in $\widehat\G$ by $\xi_\pm$\,; these are square integrable irreducible representations, having positive Plancherel measure.
The corresponding Duflo-Moore operators are 
$$
(D_{\xi_\pm}\varphi)(s) = \left\vert s \right\vert\varphi(s)=\pm s\varphi(s)\,,\quad\pm s>0\,.
$$

For a symbol $A\in\mathscr B^\oplus_2(\G\times\wG)\cong L^2(\G)\oplus L^2(\G)$\,, formula (\ref{eq:st}) reads
$$
[ \op(A)u](b,a) =\sum_\pm\int_{\mathbb R}\!\int_{\mathbb R_+}\!\frac 1 {(a')^{3/2} } \mathrm{Tr}\left(A_\pm(b,a)D_\pm^{\frac 1 2}\pi_\pm\!\left( b - \frac a {a'} b'\!,\frac a {a'}\right)\!^*\right) u(b',a') \, db'\!da'.
$$
Since in this case $\widehat{\G}$ is measurably equivalent to a two-point space, our symbols can be seen as pairs of functions depending only on the group variable $(b,a)\in \G$\,. However, since these symbols have operator values in $\H_\pm=L^2(\R_\pm)$\,, they generate a non-commutative calculus; to get the commutative calculus of multiplication operators one has to restrict to scalar valued symbols only depending on the variable in $\G$\,, that are very particular.

\begin{remark}
Let $Y =(0,1)$ and $Z= (1,0)$ be the generators of $\mathfrak g$\,; they satisfy the commutation relation $[Y,Z] = Z$\,. If $d\pi_+(X)\varphi = \frac {d}{dt}\pi_+[{\sf
exp}(tX)]\varphi\!\mid_{t=0}$ gives the (densely defined) induced representation of $\mathfrak g$ on $L^2(\mathbb R)$\,, then
$$
[d\pi_+(Y)\varphi](s) = \frac 1 2 \varphi(s) + s\varphi'(s)\,,\quad[d\pi_+(Z)\varphi](s) = 2 \pi i s
\varphi(s)\,.
$$
Clearly $[d\pi_+(Y),d\pi_+(Z)] = d\pi_+(Z)$\,. Note that
$$
id\pi_+(Y) = \frac i 2 \left( s \cdot \frac d {ds } + \frac d { ds} \cdot
  s\right)
$$
is (formally) the infinitesimal generator of dilations of $\,\mathbb R_+$\,. Similar statements hold for the sign $-$\,.
\end{remark}

\begin{acknowledgement}
M. Sandoval has been supported by Beca de Magister Nacional 2016 Conicyt and partially supported by N\'ucleo Milenio de F\'{\i}sica Matem\'atica RC120002.
M. M\u{a}ntoiu is supported by the Fondecyt Project 1120300.
\end{acknowledgement}

\end{document}